\documentclass[12pt]{amsart}
\usepackage{amsmath,amssymb,enumerate}
\newtheorem{theorem}{Theorem}[section]
\newtheorem{claim}{}[theorem]
\newtheorem{lemma}[theorem]{Lemma}

\newtheorem{conjecture}[theorem]{Conjecture}
\theoremstyle{definition}

\newcommand{\ceil}[1]{\left\lceil #1 \right\rceil}
\newcommand{\floor}[1]{\left\lfloor #1 \right\rfloor}

\newcommand{\cN}{\mathcal{N}}

\newcommand{\cP}{\mathcal{P}}

\newcommand{\cX}{\mathcal{X}}

\newcommand{\wh}{\widehat}

\newcommand{\con}{/}
\newcommand{\del}{\backslash}
\newcommand{\elem}{\varepsilon}

\DeclareMathOperator{\si}{si}

\DeclareMathOperator{\PG}{PG}
\DeclareMathOperator{\cl}{cl}

\DeclareMathOperator{\AG}{AG}

\newcommand{\bb}[1]{\mathbb{#1}}

\newenvironment{subproof}[1][\proofname]{%
	\begin{proof}[Subproof:]%
	}{%
	\end{proof}%
}
\allowdisplaybreaks[1]
\author[Nelson]{Peter Nelson}
\author[Norin]{Sergey Norin}
\title{The smallest matroids with no large independent flat}
\begin{document}
\maketitle
\begin{abstract}
	We show that a simple rank-$r$ matroid with no $(t+1)$-element independent flat has at least as many elements as the matroid $M_{r,t}$ defined to be the direct sum of $t$ binary projective geometries whose ranks pairwise differ by at most $1$. We also show for $r \ge 2t$ that $M_{r,t}$ is the unique example for which equality holds.
\end{abstract}
\section{Introduction}

Call a set $S$ in a matroid $M$ a \emph{claw} of $M$ if $S$ is both a flat and an independent set of $M$. A \emph{$k$-claw} is a claw of size $k$. These objects were introduced by Bonamy et al. [\ref{bkknp}] and studied by Nelson and Nomoto [\ref{cf}]; both of these papers consider the structure of $3$-claw-free binary matroids. 

Here we deal with general matroids, and address the simple extremal question of determining the smallest simple rank-$r$ matroids omitting a given claw; we solve this problem and characterize the tight examples.  

Theorem 1.8 of [\ref{cf}] shows that, for $r \ge 4$, the unique smallest simple rank-$r$ binary matroid with no $3$-claw is the direct sum of two binary projective geometries of ranks $\floor{r/2}$ and $\ceil{r/2}$. We show that, perhaps surprisingly, the exact same construction is also extremal for general matroids, and that its natural generalization is still extremal for excluding larger claws. For integers $r \ge 1$ and $t \ge 1$, let $M_{r,t}$ denote the matroid that is the direct sum of $t$ (possibly empty) binary projective geometries, whose ranks sum to $r$ and pairwise differ by at most $1$. We prove the following, which was conjectured for the special case of binary matroids in [\ref{cf}].

\begin{theorem}\label{main}
	Let $r,t \ge 1$ be integers. If $M$ is a simple rank-$r$ matroid with no $(t+1)$-claw, then $|M| \ge |M_{r,t}|$. If equality holds and $r \ge 2t$, then $M \cong M_{r,t}$. 
\end{theorem}

Note that for $r \le t$, the matroid $M_{r,t}$ is free and therefore the theorem is trivial. For $t < r < 2t$, there is a rather tame family of  exceptional tight examples, which we describe in Theorem~\ref{maintech}. 
One can also ask a similar question with `simple' relaxed to `loopless'. (One must still insist that $M$ is loopless, since any matroid with a loop has no claw.) In this case the answer is much less interesting; the direct sum of $r-t$ parallel pairs and $t$ coloops has $2r-t$ elements and is the unique smallest loopless rank-$r$ matroid with no $(t+1)$-claw, this is a consequence of Lemma~\ref{lowrank} below. 

\vskip 5pt

\noindent {\bf Graph Theory.}
The study of the structure of $3$-claw-free binary matroids in [\ref{bkknp},\ref{cf}] was motivated by structural results in graph theory. In the context of these works, the graph-theoretic notions of induced subgraphs, cliques, chromatic number and forests have analogies in the setting of simple binary matroids: cliques are analogous to projective geometries in the sense of being maximal with a given rank, while claws correspond to induced forests. A graph-theoretic analogue of Theorem~\ref{main} using this correspondence would characterize graphs on $r$ vertices  with minimum number of edges and no induced forests of given size. From the matroidal point of view, the natural measure of the size of a forest is the number of edges, but there seem to exist no direct analogue of Theorem~\ref{main} using this measure. Defining the size of a forest as the number of its vertices works much better, as follows. 

Let $G_{n,t}$ denote the graph on $n$ vertices which is a disjoint union of $t$ complete subgraphs, whose sizes pairwise differ by at most one.  Tur\'an's classical theorem ~[\ref{turan}] is equivalent to the statement that $|E(G)| \geq |E(G_{n,t})|$ for every graph $G$ on $n$ vertices with no stable set of size $t+1$. This observation implies that the following graph-theoretic analogue of Theorem~\ref{main} generalizes Tur\'an's theorem for $n \geq 3t$. 

\begin{theorem}\label{graph} Let $n,t \ge 1$ be integers such that $n \geq 3t$.
	If $G$ is a graph on $n$ vertices having no forest on $2t+1$ vertices as an induced subgraph, then $|E(G)| \geq |E(G_{n,t})|$ and  the equality holds only if $G$ 
	If equality holds then $G \cong G_{n,t}$. 
\end{theorem}

We give a short proof of Theorem~\ref{graph}, obtained by adapting one of the standard proofs of Turan's theorem, in Section~\ref{s:graph}.
 
\vskip 5pt

\noindent {\bf Triangle-free matroids.}
The extremal examples in Theorem~\ref{main} have many triangles, and our proof techniques analyze triangles closely. It seems plausible that if $M$ is required to be triangle-free, then the sparsest examples, instead of projective geometries, come from binary affine geometries, which are triangle-free and have $2$-claws but no $3$-claws. (An affine geometry $\AG(r-1,2)$ is obtained from a projective geometry $\PG(r-1,2)$ by deleting a hyperplane.) This leads us to conjecture the following.

\begin{conjecture}\label{agc}
	Let $t,r$ be integers with $t \ge 1$ and $t|r$. If $M$ is a simple triangle-free matroid with no $(2t+1)$-claw, then $|M| \ge t 2^{r/t-1}$. 
\end{conjecture} 

This conjectured bound holds with equality when $M$ is the direct sum of $t$ copies of a rank-$(r/t)$ binary affine geometry; these should be the only cases where equality holds. We prove this in the easy case where $t = 1$; see Lemma~\ref{ag}.

\vskip 5pt
In what follows, we use the notation of Oxley [\ref{oxley}]; flats of a matroid of rank $1$ and $2$ are \emph{points} and \emph{lines} respectively.  We additionally write $|M|$ for $E(M)$. A \emph{simplification} of $M$ is any matroid obtained from $M$ by deleting all loops and all but one element from each parallel class. All such matroids are clearly isomorphic; we write $\si(M)$ for a generic matroid isomorphic to a simplification of $M$, and write $\elem(M)$ for $|\si(M)|$, the number of points of $M$. A family $\cX$ of sets are \emph{skew} in a matroid $M$ if $r_M(\cup \cX) = \sum_{X \in \cX} r_M(X)$, and we say that a set $X$ is \emph{skew} to a set $Y$ if $\{X,Y\}$ is a skew family: i.e. $r_M(X \cup Y) = r_M(X) + r_M(Y)$.

\section{The Bound}

In this section we give the easy proof of the lower bound in Theorem~\ref{main}. Our first lemma shows that the property of being $(t+1)$-claw-free is essentially closed under contraction; if $F$ is a $k$-claw of some simplification of $M$, call $F$ a \emph{$k$-pseudoclaw} of $M$.

\begin{lemma}\label{contract}
	Let $k \ge 1$. If $M$ is a simple matroid and $X \subseteq E(M)$, then every $k$-pseudoclaw of $M \con X$ is a $k$-claw of $M$. 
\end{lemma}
\begin{proof}
	Let $M' = (M \con X) \del P$ be a simplification of $M \con X$, and suppose that $M'$ has a $k$-claw $F$. Since $F$ is independent in $M \con X$, it is independent in $M$, and is skew to $X$ in $M$. Since $F$ is a flat of $M \con X$, we have $\varnothing = \cl_{M'}(F) - F = \cl_M(X \cup F) - (F \cup X \cup P) \supseteq (\cl_M(F) - F)-(X \cup P)$, giving $\cl_M(F) - F \subseteq X \cup P$.  
	
	The sets $\cl_M(F)$ and $X$ are skew in $M$, so $\cl_M(F)-F \subseteq P$. Suppose that $e \in (\cl_M(F)-F) \cap P$. Then there exists $e' \in E(M')$ parallel to $e$ in $M \con X$; since $e' \in \cl_M(F)-X \subseteq \cl_{M \con X}(F)$ we also have $e \in \cl_{M \con X}(F)$ and so $e \in \cl_{M'}(F) = F$, contrary to the choice of $e$. It follows that $\cl_M(F)-F$ intersects neither $X$ nor $P$ so is empty; therefore $F$ is a $k$-claw of $M$. 
\end{proof}

Let $f(r,t) = |M_{r,t}|$. Since a rank-$n$ projective geometry has $2^n-1$ elements, we clearly have $f(r,t) = (t-a) 2^{\floor{r/t}} + a 2^{\ceil{r/t}} - t$, where $a \in \{0, \dotsc, t-1\}$ is the integer with $a \equiv r \pmod{t}$. More importantly for our purposes, we can define $f$ recursively; it is easy to check that $f(r,t) = r$ for all $0 \le r \le t$ and $f(r,t) = 2f(r-t,t) + t$ for $r > t$. We use this recurrence and the previous lemma to prove the lower bound in our main theorem.

\begin{theorem}\label{mainbound}
	If $t \ge 1$ is an integer and $M$ is a simple rank-$r$ matroid with no $(t+1)$-claw, then $|M| \ge f(r,t)$. 
\end{theorem}
\begin{proof}
	Let $M$ be a counterexample for which $r + |M|$ is minimized. If $M$ is a free matroid then clearly $r \le t$, in which case $f(r,t) = r \le |M|$ so $M$ is not a counterexample. Therefore $M$ has a non-coloop $e$. 
	
	Since $|M \del e| < |M| \le f(r,t)$ but $M \del e$ is not a counterexample, there must be a $(t+1)$-claw $S'$ in $M \del e$. Now the matroid $M|\cl_M(S')$ has rank $t+1$ and has at most $t+2$ elements, so has at most one circuit. There is thus a $t$-element subset $S$ of $\cl_M(S')$ containing at most $|C|-2$ elements of each $C$ circuit of $M$; this set $S$ is a $t$-claw. 
	
	If there is some rank-$(t+1)$ flat $F$ containing $S$ for which $|F-S|= 1$, then $F$ is a $(t+1)$-claw. Therefore every such flat satisfies $|F-S| \ge 2$; since $S$ is a flat, it follows that every parallel class of $M \con S$ has size at least $2$. Moreover, $\si(M \con S)$ is a rank-$(r-t)$ matroid that by Lemma~\ref{contract} has no $(t+1)$-claw; inductively we have $|\si(M \con S)| \ge f(r-t,t)$. Now 
	\[|M| = |S| + |M \con S| \ge t + 2|\si(M \con S)| \ge t + 2f(r-t,t) = f(r,t),\] as required. 
\end{proof}

\section{Equality}

In this section we characterize matroids for which the bound in Theorem~\ref{mainbound} holds with equality. We require two lemmas; the first (which uses Tutte's characterization of binary matroids as those with no $U_{2,4}$-minor~[\ref{tutte}])  corresponds to the case $t = 1$ of Theorem~\ref{mainbound}. 

\begin{lemma}\label{base}
	If $M$ is a simple rank-$r$ matroid with no $2$-claw, then $|M| \ge 2^r-1$. If equality holds then $M \cong \PG(r-1,2)$. 
\end{lemma}
\begin{proof}
	Let $M$ be a minor-minimal counterexample. Clearly $r(M) \ge 3$. Let $e \in E(M)$ and let $H$ be a hyperplane of $M$ not containing $e$. Since $M|H$ has no $2$-claw, we have $|H| \ge 2^{r-1}-1$ by the minimality of $M$. For each $x \in H$, the line spanned by $x$ and $e$ contains an element of $E(M)-\{e,x\}$, and these lines pairwise intersect only in $e$, so we see that $|M| \ge 2|H| + 1 \ge 2^r-1$ as required. 
	
	If $|M| = 2^r-1$, then equality holds above, so $|H| = 2^{r-1}-1$ and thus $M|H \cong \PG(r-2,2)$. Moreover, for each $x \in H$ we have $|\cl_M(\{e,x\})| = 3$ and $E(M) = \cup_{x \in H}(\cl_M(\{e,x\}))$, which implies that $\si(M \con e) \cong M|H \cong \PG(r-2,2)$ so $M \con e$ is binary. The choice of $e$ was arbitrary, so $M \con e$ is binary for all $e$; since $r \ge 3$ this gives that $M$ has no $U_{2,4}$-minor so is binary. Since $M$ is simple with $2^r-1$ elements, this implies  $M \cong \PG(r-1,2)$, a contradiction. 
\end{proof}

Note that if $t < r < 2t$ then $|M_{r,t}| = 2r-t$. In this range, the matroids $M_{r,t}$ are not the only ones satisfying the bound in Theorem~\ref{mainbound} with equality. The other examples include direct sums of circuits and coloops, and the matroid $M_{r,t}$ is the special case where all these circuits are triangles. The following lemma shows that these are the only examples. It also implies the characterization of the smallest $(t+1)$-claw-free matroids that are not required to be simple that was claimed in the introduction. 

\begin{lemma}\label{lowrank}
	Let  $r \ge t \ge 1$ be integers. If $M$ is a loopless rank-$r$ matroid with no $(t+1)$-claw, then $|M| \ge 2r - t$. If equality holds, then $M$ is the direct sum of $r-t$ circuits and some number of coloops. 
\end{lemma}
\begin{proof}
	
	 Suppose first that every circuit of $M^*$ has at most two elements. Then, since $M^*$ is coloopless, it is the direct sum of loops and parallel classes of size at least two; let $\cP$ be its set of parallel classes, so $r(M^*) = |\cP|$ and $r = |M| - |\cP|$. Let $U$ be a set comprising exactly two elements from each $P \in \cP$. Since $r(M^*|U) = r(M^*)$, the set $(E-U)$ is independent in $M$, and since $M^*|U$ is coloopless, the set $(E-U)$ is also a flat of $M$, so is a claw of $M$. By hypothesis, it follows that $t \ge |E - U| = |M| - 2|\cP|$. Therefore $2r - t \le 2(|M| - |\cP|) - (|M| - 2|\cP|) = |M|$, as required. If equality holds, then $t = |M| - 2|\cP| = r - |\cP|$, so $|\cP| = r - t$. By the definition of $\cP$, each $P \in \cP$ is a circuit of $M$, and each other element of $M$ is a coloop; this gives the required structure. 
	
	
	
	We may therefore assume $M^*$ has a circuit $C$ of size at least $3$. Let $B$ be a basis of $M^*$ containing all but one element of $C$. Since $M^*$ has no coloops, for each $x \in B$, there is a circuit $C_x$ of $M^*$ for which $x \in C_x$ and $|C_x \cap B| = |C_x|-1$; choose the $C_x$ so that $C_x = C$ for each $x \in C$. Let $X = \cup_{x \in X} C_x$. Since each $C_x$ contains only one element outside $B$ and the element of $C_0-B$ is chosen at least twice, we have $|X| < 2r(M^*)$. 
	
	By construction, the set $X$ contains a basis and, since $X$ is a union of circuits of $M^*$, the matroid $M^*|X$ has no coloops. Let $Y = E(M)-X$; by construction the set $Y$ is independent in $M$, and $M/Y$ has no loops, so $Y$ is a flat, and thus a claw, of $M$. Hence $|Y| \le t$ and so $|X| \ge |M|-t$. By our upper bound on $|X|$, this gives $|M| - t < 2r(M^*) = 2(|M|-r)$ and so $|M| > 2r-t$, as required. 
\end{proof}	

We are now ready to strengthen Theorem~\ref{mainbound} with an equality characterisation. Note that both outcomes n the equality case imply that $M$ has a $t$-claw. 

\begin{theorem}\label{maintech}
	Let $t \ge 1$. If $M$ is a simple rank-$r$ matroid with no $(t+1)$-claw, then $|M| \ge f(r,t)$. If equality holds, then either 
	\begin{itemize}
		\item $M \cong M_{r,t}$, or 
		\item $t < r < 2t$ and $M$ is the direct sum of coloops and exactly $r-t$ circuits, not all of which are triangles.
	\end{itemize}

\end{theorem}
\begin{proof}

	Consider a counterexample $M$ for which $|M| + r$ is minimized. Clearly $r > t$, as otherwise $|M| \ge r = f(r,t)$ and there is nothing to prove. Therefore $M$ is not a free matroid, since otherwise any $(t+1)$-element subset of $E(M)$ is a claw. 
	
	\begin{claim}
		$M$ has a $t$-claw.
	\end{claim}
	\begin{subproof}
		Let $e$ be a non-coloop of $M$; by the minimality of $M$, the matroid $M \del e$ has a $(t+1)$-claw $F$; now $M|\cl_M(F)$ has rank at least $t+1$ and has at most $t+2$ elements, so has at most one circuit. There is thus a $t$-element subset of $\cl_M(F)$ that contains at most $|C|-2$ elements of each circuit $C$ of $M|\cl_M(F)$; this set is a $t$-claw of $M$. 
	\end{subproof}

  	Call a $t$-claw $S$ of $M$ \emph{generic} if no four-point line of $M$ intersects $S$, and  exactly $f(r-t,t)$ triangles of $M$ intersect $S$. Let $S$ be a $t$-claw of $M$, chosen \emph{not} to  be generic if such a choice is possible. 
 	
 	\begin{claim}
 		$|M| = f(r,t)$, each parallel class of $M \con S$ has size $2$, and $\elem(M \con S) = f(r-t,t)$. 
 	\end{claim}
 	\begin{subproof}
 		The matroid $M \con S$ has rank $r-t$ and, by Lemma~\ref{contract}, has no $(t+1)$-pseudoclaw. Therefore $\si(M \con S)$ has no $(t+1)$-claw, so $\elem(M \con S) \ge f(r-t,t)$. Moreover, if some parallel class $Y$ of $M \con S$ has size $1$, then $S \cup Y$ is a $(t+1)$-claw of $M$, so every parallel class of $M \con S$ has size at least $2$, giving $|M \con S| \ge 2 \elem(M \con S)$.  Therefore
 		\[f(r,t) \ge |M| \ge 2\elem(M \con S) + |S| \ge 2f(r-t,t) + t = f(r,t). \]
		Equality holds throughout, which gives the claim. 
 	\end{subproof}
 
 	The matroid $\si(M \con S)$ has no $(t+1)$-claw and has $f(r-t,t)$ elements, so inductively satisfies one of the conclusions of the theorem. For each component $N$ of $M \con S$, the matroid $\si(N)$ is either a circuit or a binary projective geometry.
 	
 	\begin{claim}\label{pclaws}
 		Let $e_1,e_2 \in E(M \con S)$. If $e_1$ and $e_2$ are in different components, then $M \con S$ has a $t$-pseudoclaw containing $e_1$ and $e_2$. If $e_1$ and $e_2$ are in the same component, then there is a $(t-1)$-element set $U$ such that $U \cup \{e_1\}$ and $U \cup \{e_2\}$ are both $t$-pseudoclaws of $M \con S$. 
 	\end{claim}
 	\begin{subproof}
 		
 		We first argue that $M\con S$ has a $t$-pseudoclaw. Since $\si(M \con S)$ satisfies one of the outcomes of the theorem, this only fails if $\si(M \con S) \cong M_{r-t,t}$ and $r-t < t$. If this holds then $|M| = f(r,t) = 2r-t$ and $M$ satisfies the hypothesis of Lemma~\ref{lowrank}, so is the direct sum of coloops and $r-t$ circuits. If these circuits are all triangles then $M \cong M_{r,t}$, and otherwise $M$ satisfies the second outcome of the theorem; both are contrary to the choice of $M$ as a counterexample. Therefore $M \con S$ has a $t$-pseudoclaw. 
 		
 		Since every component of $\si(M \con S)$ is a circuit or projective geometry, given any pseudoclaw $K$ of $M$ and any $e,e'$ in the same component of $M \con S$ for which $e \in K$ and $e' \notin K$, the set $(K - e) \cup \{e'\}$ is also a pseudoclaw. Since $M \con S$ has at least one $t$-pseudoclaw, both conclusions of the claim easily follow. 
 	\end{subproof}
 	
	The above claim implies in particular that every element of $M \con S$ is in a $t$-pseudoclaw.
 	
 	\begin{claim}\label{bij}
 		For each $t$-pseudoclaw $U$ of $M \con S$, there is a bijection $\psi_{U}$ from $U$ to $S$ so that for each $e \in U$, the flat $T_e = \cl_M(e,\psi_{U}(e))$ is a triangle of $M$, and so that $M|\cl_M(S \cup U) = \oplus_{e \in U} (M|T_e)$. 
 	\end{claim}
 	\begin{subproof}
		Since the closure of $U$ in $M \con S$ is obtained from $U$ by extending each element of $U$ once in parallel, we have $|\cl_M(S \cup U)| = |S| + 2|U| = 3t$. By Lemma~\ref{lowrank}, it follows that the simple rank-$2t$ matroid $M' = M|\cl_M(S \cup U)$ is the direct sum of $t$ circuits and some set of coloops, and therefore that is precisely the direct sum of $t$ triangles. Since $S$ is a $t$-claw of $M'$ and $U$ is a $t$-pseudoclaw of $M' \con S$, both $S$ and $U$ must be transversals of this set of triangles. The claim follows. 
 	\end{subproof}
 
	Every element $e$ of $M \con S$ is contained in a $t$-pseudoclaw, so the above claim implies that each such $e$ is in exactly one triangle that intersects $S$. Write $\psi(e)$ for the unique element of $S$ for which $e$ and $\psi(e)$ are contained in a triangle; we have $\psi_U(e) = \psi(e)$ for each $t$-pseudoclaw $U$ of $M \con S$ containing $e$. 
 	
 	Since $S$ is a claw and no rank-$1$ flat of $M \con S$ has more than two elements, no line of $M$ that intersects $S$ has more than three elements. Moreover, each $e \in E(M \con S)$ is in exactly one triangle of $M$ that intersects $S$, so the number of triangles of $M$ that intersect $S$ is exactly $\tfrac{1}{2}|M \con S| = \tfrac{1}{2}(2\elem(M \con S)) = f(r-t,t)$. Therefore $S$ is generic. It follows from the choice of $S$ that every $t$-claw of $M$ is generic.

 	\begin{claim}\label{nicetriangles}
 		For all $e_1, e_2 \in E(M \con S)$, we have $\psi(e_1) = \psi(e_2)$ if and only if $e_1$ and $e_2$ are in the same component of $M \con S$. Moreover, $M \con S$ has exactly $t$ components.
 	\end{claim}
 	\begin{subproof}
 		Suppose that $e_1$ and $e_2$ are in the same component of $M \con S$. By~\ref{pclaws}, there is a set $U \subseteq E(M \con S)$ such that $U \cup \{e_1\}$ and $U \cup \{e_2\}$ are both $t$-pseudoclaws of $M \con S$. For each $i \in \{1,2\}$, there is a bijection $\psi_i = \psi_{U \cup \{e_i\}}$ from $U \cup \{e_i\}$ to $S$, and moreover for each $e \in U$ we have $\psi_1(e) = \psi(e) = \psi_{2}(e)$. Therefore $\psi_1$ and $\psi_2$ agree on all $t-1$ elements of $U$; thus $\psi_{1}(e_1) = \psi_{1}(e_2)$ and so $\psi(e_1) = \psi(e_2)$. \
 		
		Suppose now that $e_1$ and $e_2$ are in different components of $M \con S$. By~\ref{pclaws} there is a $t$-pseudoclaw $U$ containing $e_1$ and $e_2$. Since $\psi_U$ is a bijection we have $\psi(e_1) = \psi_U(e_1) \ne \psi_U(e_2) = \psi(e_2)$, as required.
		
		It follows from the first part that the image of $\psi$ has size equal to the number of components of $M \con S$. But clearly the image of $\psi$ contains the image of $\psi_U$, which is equal to $S$, for each $t$-pseudoclaw $U$. Therefore $\psi$ has image $S$, so $M \con S$ has exactly $|S| = t$ components. 
 	\end{subproof}
 
 		
 
 	Let $\cN$ be the set of components of $M \con S$. By~\ref{nicetriangles}, for each $N \in \cN$ there is some $\psi(N)$ for which $\psi(e) = \psi(N)$ for each $e \in E(N)$. Since $|S| = |\cN| = t$, the $t$-tuple $\psi(N) : N \in \cN$ is a permutation of $S$. For each $N \in \cN$, let $\wh{N} = M|(E(N) \cup \psi(N))$.
 	
	\begin{claim}\label{shortlines}
		If $N \in \cN$ and $L$ is a line of $M$ intersecting $E(N)$, then either $|L| = 2$, or $|L| = 3$ and $L \subseteq E(\wh{N})$.  
	\end{claim}
	\begin{subproof}
		Let $U$ be a $t$-pseudoclaw of $M \con S$ containing an element $e \in E(N) \cap L$. Note that $U$ is a generic $t$-claw in $M$, which gives $|L| \le 3$.
		
		Suppose that $|L| = 3$. Note that $1 = r_{M \con S}(e) \le r_{M \con S}(L-S) \le r_M(L) = 2.$ If $r_{M \con S}(L-S) = 2$ then $L$ is a triangle of $M \con S$ that intersects the component $N$ of $M \con S$, so obviously $L \subseteq E(N)$. If $r_{M \con S}(L-S) = 1$ then $L \subseteq \cl_M(S \cup \{e\}) \subseteq \cl_M(S \cup U)$, so ~\ref{bij} gives $L = \cl_M(\{e, \psi(e)\})$. Since $L-\psi(e)$ is a two-element rank-$1$ set in $M \con S$, the third element of $L$ is the element of $N$ parallel to $e$, so $L \subseteq E(\wh{N})$ as required. 
	\end{subproof}

	For each $N \in \cN$ and $e \in E(N)$, let $\tau(e)$ be the number of $3$-element lines of $M$ containing $e$, and let $\sigma(e)$ be the number of elements of $E(\wh{N} \del e)$ that are \emph{not} in a $3$-element line of $M$ with $e$. By the previous claim we have $ 2\tau(e) + \sigma(e) = |\wh{N} \del e| = |N|$.

	\begin{claim}\label{npg}
		For each $N \in \cN$, every line of $E(\wh{N})$ has size $3$.
	\end{claim}
	\begin{subproof}
		Suppose not, so there is some $e \in E(N)$ for which $\sigma(e) > 0$. Let $U$ be a $t$-pseudoclaw of $M \con S$ containing $e$. Since $U$ is a generic $t$-claw of $M$, we have $\sum_{u \in U} \tau(u) = f(r-t,t)$, so 
		\begin{align*}
			|M| &= |S| + \sum_{N \in \cN} |N| \\
			&= t + \sum_{u \in U}(2\tau(u) + \sigma(u))\\
			&= t + 2f(r-t,t) + \sum_{u \in U} \sigma(u)\\ 
			&\ge f(r,t) + \sigma(e).
		\end{align*}
		Since $|M| = f(r,t)$ and $\sigma(e) > 0$, this is a contradiction.
	\end{subproof}
	Let $N \in \cN$. It is clear, since $N$ is obtained from $\wh{N} \con \psi(N)$ by $t-1$ successive extension-contraction operations, that $r(N) \le r(\wh{N})-1$. The matroid $\si(N)$ is a circuit or a binary projective geometry, so 
	\[|\wh{N}| = |N| + 1 \le 2(2^{r(N)}-1)+1 = 2^{r(N)+1}-1 \le 2^{r(\wh{N})}-1.\]
	By~\ref{npg} and Lemma~\ref{base} we have $|\wh{N}| \ge 2^{r(\wh{N})}-1$, so equality holds and therefore each matroid $\wh{N}$ is a binary projective geometry of rank $r(N) + 1$. The sets $E(\wh{N}): N \in \cN$ partition $E(M)$, so 
	\[r \le \sum_{N \in \cN} r(\wh{N}) = \sum_{N \in \cN}(r(N) + 1) = r(M \con S) + |\cN|  = t + r(M \con S) = r,\]
	so equality holds throughout, and the sets $E(\wh{N})$ are mutually skew in $M$. Thus $M$ is the direct sum of $t$ nonempty binary projective geometries. If $M$ has components of ranks $r_1,r_2$ with $r_2 \ge r_1 + 2$, then deleting both and replacing them with projective geometries of rank $r_2-1$ and $r_1 + 1$ respectively gives a matroid $M'$ with no $(t+1)$-claw satisfying 
	\[|M|-|M'| = 2^{r_2}+2^{r_1} - 2^{r_2-1} - 2^{r_1 + 1} = 2^{r_2-1} - 2^{r_1+1} > 0,\]
	which contradicts the minimality of $|M|$. It follows that no two components of $M$ have ranks differing by more than $1$, so $M \cong M_{r,t}$, contrary to the choice of $M$ as a counterexample.
\end{proof}

Finally, we prove the $t = 1$ case of Conjecture~\ref{agc} as promised.

\begin{lemma}\label{ag}
	If $M$ is a simple rank-$r$ triangle-free matroid with no $3$-claw, then $|M| \ge 2^{r-1}$. If equality holds, then $M \cong AG(r-1,2)$. 
\end{lemma}
\begin{proof}
	We may assume that $r \ge 3$. We first show that every triple of distinct elements of $M$ is contained in a four-element circuit; indeed, given such a triple $I$, since $I$ is not a triangle or a $3$-claw, we have $r_M(I) = 3$ and $\cl_M(I) \ne I$. Thus there is some $x \in \cl_M(I)-I$. Since $M$ is triangle-free, no pair of elements of $I$ spans $x$, so $I \cup \{x\}$ is a $4$-element circuit. 
	
	Let $e \in E(M)$. Since $M$ is triangle-free, the matroid $M \con e$ is simple. If $M \con e$ has a $2$-claw $I$, then $I \cup \{e\}$ is clearly a $3$-claw of $M$; therefore $M \con e$ is $2$-claw-free and so $|M \con e| \ge 2^{r-1}-1$ by Lemma~\ref{base}. It follows that $|M| \ge 2^{r-1}$ as required. 
	
	If equality holds, then $M \con e \cong \PG(r-2,2)$ so $M \con e$ is binary. This holds for arbitrary $e \in E(M)$; it follows (since $r \ge 3$) that $M$ has no $U_{2,4}$-minor so is also binary. A simple rank-$r$ triangle-free binary matroid has at most $2^{r-1}$ elements and equality holds only for binary affine geometries (see [\ref{bb}], for example); therefore $M \cong \AG(r-1,2)$. 
\end{proof}

\section{Graphs}\label{s:graph}

Let $g(n,t): \bb{N} \times \bb{N} \to \bb{N}$ be defined recursively by \begin{itemize}
\item $g(n,t) = 0$ for $n < 2t$,  
 \item $g(n,t) = 3(n-2t)$  for  $2t \le n \le 4t$,  
\item  $g(n,t) = g(n-1,t) + \left \lceil \frac{n}{t} \right \rceil -1$ ,  for $n > 4t$.
\end{itemize}

It is easy to check that $|E(G_{n,t})|=g(n,t)$ for $n \geq 3t$ (although not for smaller $n$). The recursion for $n > 4t$ in fact also holds when $ 3t < n \le 4t$. Thus the next theorem implies Theorem~\ref{graph}. Like in the matroid setting, there is a range of cases ($2t \le n < 4t$) with non-unique tight examples, while the tight examples are just the expected ones for larger $n$.

\begin{theorem}\label{graph2eq} Let $n,t \ge 1$ be integers.
	Let $G$  be a simple graph on $n$ vertices such that no forest on $2t+1$ vertices is an induced subgraph of $G$. Then $|E(G)| \geq g(n,t)$.
	
	 If equality holds and $n < 4t$, then every component of $G$ is a complete graph on $1,3$ or $4$ vertices. If equality holds and $n \ge 4t$, then $G \cong G_{n,t}$. 
\end{theorem}

\begin{proof}
	We prove the theorem by induction on $|V(G)|$. We may clearly assume that $n \geq 2t+1$, as otherwise the result is easy. Let $v$ be a vertex of $G$ of maximum degree. 
	
	If $\deg(v) \leq 2$, every component of $G$ is a path or a cycle. Let $S$ be the set of vertices of cycles of $G$, and $b$ be the number of cycles of $G$; note that $b \le \tfrac{1}{3}|S| \le \tfrac{1}{3}n$. Clearly $G$ contains an induced forest on $n - b$ vertices, so $n-b \le 2t$. This gives $n \le 2t + \tfrac{n}{3}$, so $n \le 3t$, which in turn implies that $g(n,t) = 3(n-2t) \le 3b$. On the other hand, we have $|E(G[S])| = |S|$, so 
	\[|E(G)| \ge |E(G[S])| = |S| \ge 3b \ge g(n,t),\]
	giving the bound. If equality holds, then $E(G) = E(G[S])$ and $b = \tfrac{1}{3}|S|$, so every component of $G$ is an isolated vertex or triangle. We have argued that $n \le 3t$; thus $G$ has the claimed structure. We may therefore assume that $\deg(v) \ge 3$. 
	
	Let $X \subseteq V(G)$ be maximal so that $G[X]$ is a forest, so  $|X| \leq 2t$. Let $Z$ be the set of non-isolated vertices of $G[X]$. As $G[X \cup \{w\}]$ contains a cycle for every $w \in V(G) \setminus X$, every such $w$ has at least two neighbors in $Z$.
	Thus \[\sum_{z \in Z}\deg(z) \geq |Z|+2|V(G)-X|  \geq |Z| + 2(n-2t).\] Hence there exists $z_0 \in Z$ such that $\deg(z_0) \geq 2(n-2t)/|Z| +1 \geq (n-2t)/t+1= n/t-1$; thus $\deg(v) \ge \left \lceil \frac{n}{t} \right \rceil -1$ by the choice of $v$. 
	
	By the above, we can assume that $\deg(v) \ge \max(3, \lceil \frac{n}{t} \rceil -1)$. Let $H = G - v$. It follows that 
	\[|E(G)| = |E(H)| + \deg(v) \ge g(n-1,t) + \max(3, \lceil \tfrac{n}{t} \rceil - 1) = g(n,t);\]
	the last equality is easy to check. This gives the desired bound. 
	
	Suppose now that equality holds. We have $\deg(v) = \max(3, \lceil \tfrac{n}{t} \rceil - 1)$, and $|E(H)| = g(n-1,t)$, so equality also holds for $H$. For each $d \ge 1$, let $n_d$ be the number of components of $H$ having exactly $d$ vertices. Call a component of $H$ with at least two vertices \emph{big}. 
	
	Let $U$ be a set containing all isolated vertices of $H$, and exactly two vertices from each big component of $H$, chosen so that $U$ contains as few neighbours of $v$ as possible. We argue that $|U| = 2t$. If $2t + 1 \le n \le 4t$, by induction, each component of $H$ is a complete graph on $1$, $3$ or $4$ vertices. It follows that $n-1 = n_1 + 3n_3 + 4n_4$ and $3(n-2t) = g(n,t) = 3n_3 + 6n_4$. This gives $|U| = n_1 + 2(n_3 + n_4) = (n-1) - \tfrac{1}{3}(3(n-2t)) = 2t$. If $n > 4t$, then $H \cong G_{n-1,t}$, so $H$ has $t$ components, all big; thus $|U| = 2t$.
	
	If each big component of $H$ contains a non-neighbour of $v$, then $U \cup \{v\}$ induces a forest on $2t+1$ vertices, a contradiction. Therefore $H$ has a big component $C$ with $V(C) \subseteq N(v)$. Induction implies that each big component of $H$ has at least $\max(3, \lceil \frac{n}{t} \rceil -1) = \deg(v)$ vertices; it follows that $|V(C)| = \deg(v) = \max(3, \lceil \frac{n}{t} \rceil -1)$, and that $G$ is obtained from $H$ by adding a new vertex with neighbourhood $V(C)$. 
	
	If $2t + 1 \le n \le 4t$, then $|V(C)| = \deg(v) = \max(3, \lceil \frac{n}{t} \rceil -1) = 3$, so $G$ is obtained from $H$ by turning a triangle into a $K_4$; thus, every component of $G$ is complete with $1$, $3$ or $4$ vertices. If $n < 4t$ then this implies that $M$ has the claimed structure, and if $n = 4t$ then the fact that $|E(G)| = g(4t,t) = 6t = \tfrac{3}{2}|V(G)|$ and $G$ has maximum degree $3$ implies that $G$ is $3$-regular, and so $G \cong G_{4t,t}$, as required. 
	
	If $n > 4t$, then $H \cong G_{n-1,t}$, then $|V(C)| = \deg(v) = \lceil \tfrac{n}{t} \rceil - 1 = \lfloor \tfrac{n-1}{t} \rfloor$, so $C$ is a smallest component of $H$. Thus $G \cong G_{n,t}$, as required. 	
\end{proof}	 
 
\section*{References}
\newcounter{refs}
\begin{list}{[\arabic{refs}]}
	{\usecounter{refs}\setlength{\leftmargin}{10mm}\setlength{\itemsep}{0mm}}
	
	\item\label{bb}
	R. C. Bose, R. C. Burton, 
	A characterization of flat spaces in a finite geometry and the uniqueness of the Hamming and the MacDonald codes, 
	\emph{J. Combin. Theory} 1 (1966), 96--104. 
	
	\item\label{bkknp}
	M. Bonamy, F. Kardo\v{s}, T. Kelly, P. Nelson, L. Postle,
	The structure of binary matroids with no induced claw or Fano plane restriction. arxiv:1806.04188
	
	\item\label{cf}
	P. Nelson, K. Nomoto, 
	The structure of claw-free binary matroids. arxiv:1807.11543

	\item \label{oxley}
	J. G. Oxley, 
	Matroid Theory,
	\emph{Oxford University Press, New York} (2011).
	
	\item \label{turan}	
P. Tur\'{a}n,
 On an extremal problem in graph theory, 
 \emph{Matematikai \'{e}s Fizikai Lapok} (1941)
  (in Hungarian),  436–-452.
  
  	\item \label{tutte}	
 W.T. Tutte, Lectures on matroids, \emph{Journal of Research of the National Bureau of Standards} (1965), 1–-47. 
\end{list}

\end{document}